\documentclass[11pt,a4paper]{article}

\usepackage[paper=a4paper,left=30mm,right=30mm,top=25mm,bottom=25mm]{geometry}
\usepackage[T1]{fontenc}
\usepackage{lmodern}
\usepackage{amsmath,amssymb,amsthm,mathtools}
\usepackage{microtype}
\usepackage{enumitem}
\usepackage{graphicx}
\usepackage[font=small,labelfont=bf,labelsep=period,skip=7pt]{caption}
\usepackage[section]{placeins}
\usepackage{flafter}
\usepackage{needspace}
\usepackage{etoolbox}
\usepackage{xcolor}
\definecolor{linkblue}{RGB}{30,65,105}
\usepackage[numbers,sort&compress]{natbib}
\usepackage[colorlinks=true,linkcolor=linkblue,citecolor=linkblue,urlcolor=linkblue]{hyperref}
\hypersetup{
  pdftitle={Rank weighting and asymmetry in Blest's rank correlation: two exact regions},
  pdfauthor={Marcus Rockel},
  pdfkeywords={Copula, Spearman's rho, Symmetrized Blest coefficient, Measure of association, Sharp inequality, Asymmetry, Rearrangement inequality, Kantorovich duality}
}

\newtheorem{theorem}{Theorem}[section]
\newtheorem{lemma}[theorem]{Lemma}
\newtheorem{proposition}[theorem]{Proposition}
\newtheorem{corollary}[theorem]{Corollary}
\theoremstyle{remark}
\newtheorem{remark}[theorem]{Remark}

\AddToHook{cmd/section/before}{\Needspace{6\baselineskip}}
\AddToHook{cmd/subsection/before}{\Needspace{5\baselineskip}}
\BeforeBeginEnvironment{theorem}{\Needspace{5\baselineskip}}
\BeforeBeginEnvironment{lemma}{\Needspace{5\baselineskip}}
\BeforeBeginEnvironment{proposition}{\Needspace{5\baselineskip}}
\BeforeBeginEnvironment{corollary}{\Needspace{5\baselineskip}}
\BeforeBeginEnvironment{proof}{\Needspace{3\baselineskip}}

\newcommand{\E}{\mathbb{E}}
\newcommand{\Prob}{\mathbb{P}}
\newcommand{\R}{\mathbb{R}}
\newcommand{\de}{\,\mathrm{d}}
\newcommand{\1}{\mathbf{1}}
\newcommand{\CC}{\mathcal{C}}
\newcommand{\Gam}{\Gamma}
\newcommand{\keywords}[1]{\par\smallskip{\small\noindent\textbf{Keywords:} #1\par}}

\title{Rank weighting and asymmetry in Blest's rank correlation: two exact regions}
\author{Marcus Rockel}
\date{\today}

\begin{document}

\maketitle
\begin{center}
\small\itshape
Department of Quantitative Finance, Institute for Economics, University of Freiburg,\\
Rempartstr.~16, 79098 Freiburg, Germany\\
\upshape\href{mailto:marcus.rockel@finance.uni-freiburg.de}{\nolinkurl{marcus.rockel@finance.uni-freiburg.de}}
\end{center}

\begin{abstract}
Blest's rank correlation $\nu$ is a variant of Spearman's rho $\rho$ that weights the leading ranks of one variable more heavily, at the price that $\nu$ is not symmetric in its arguments.
We quantify both features by determining the exact region of $(\rho,\nu)$ over all bivariate copulas, as well as that of $(\eta,\nu)$, where $\eta$ is the symmetrized Blest coefficient of Genest and Plante.
The latter region is a linear image of the set of all pairs $(\nu(C),\nu(C^\top))$ formed by a copula $C$ and its transpose.
Consequently, Blest's coefficient differs from Spearman's rho by at most $1/4$, and interchanging the two variables changes it by at most $27/64$, improving on the bound $1/2$ implied by the first inequality.
For every given value of $\rho$ or $\eta$, each corresponding extreme value of $\nu$ is attained by exactly one copula, given in closed form and supported on finitely many line segments.
Near countermonotonicity, the upper extremizers of the $(\eta,\nu)$-region are supported on the graph of a function of the second coordinate, yet their conditional laws given the first coordinate carry two atoms.
The proofs rest on a rearrangement inequality with equality case and on explicit Kantorovich potentials.
\end{abstract}

\par\smallskip{\small\noindent\textbf{MSC 2020:} 62H05; 62H20; 60E15; 49Q22\par}

\section{Introduction}

Spearman's rho \cite{spearman1904proof} and Kendall's tau \cite{kendall1938new} measure the agreement between two rankings of the same items, irrespective of the scales on which the underlying quantities are measured, and both treat every position in a ranking alike: exchanging the two leading items costs exactly as much as exchanging the two trailing ones.
Often, however, the importance of a disagreement depends on where in the ranking it occurs.
When judges rank competitors, agreement on the leading places counts for more than agreement near the bottom of the field, and it was subjectively judged Olympic events that led Blest \cite{blest2000rank} to a coefficient giving greater weight to the first ranks.
His coefficient adapts Spearman's rho so that a disagreement about the leading places changes the value more than the same disagreement further down.
Blest's coefficient thus records not only how strongly two rankings agree but also where they disagree.

\Needspace{7\baselineskip}
The emphasis of Blest's coefficient comes from weighting the ranks of one of the two variables, and, as a consequence, the value generally changes when the two rankings exchange roles \cite{genest2003blest}.
These two features raise two quantitative questions:
\begin{enumerate}[label=\textup{(Q\arabic*)},ref=(Q\arabic*),leftmargin=3.2em,topsep=0.8ex,itemsep=0.4ex,before=\itshape,after={\endgraf\vspace{0.6ex}\vspace{0pt}}]
\item\label{q:weighting} How much can the extra weight on the leading ranks move the coefficient away from its unweighted counterpart?
\item\label{q:transpose} How much can the coefficient move when the two variables are interchanged?
\end{enumerate}
Exact regions answer these questions at the population level: the exact region of two coefficients consists of every pair of values that the two attain jointly over all dependence structures.
We call its vertical sections \emph{fibres}, so that the fibre over a given value of the first coefficient collects all compatible values of the second.
From the boundary of an exact region one reads off the sharp inequalities between the two coefficients, and the dependence structures sitting on the boundary reveal what drives the coefficients furthest apart.
In this paper, we determine two exact regions for Blest's coefficient, one for each question.

To make this precise, let $\CC$ denote the class of bivariate copulas, that is, distribution functions on $[0,1]^2$ with uniform marginals, and we refer to \cite{nelsen2006introduction,durante2016principles} for background.
The copulas $M(u,v)\coloneqq\min\{u,v\}$, $W(u,v)\coloneqq\max\{u+v-1,0\}$, and $\Pi(u,v)\coloneqq uv$ describe comonotonicity, countermonotonicity, and independence, respectively.
For a copula $C$, the population version of Blest's rank correlation is
\begin{equation}
  \label{eq:nu-def}
  \nu(C)=24\int_0^1\!\int_0^1(1-u)\,C(u,v)\de u\de v-2,
\end{equation}
as given in \cite[Sec.~3]{genest2003blest}, a normalization that makes $\nu(M)=1$, $\nu(W)=-1$, and $\nu(\Pi)=0$.
Up to its additive normalizing constant, $\nu$ arises from Spearman's rho,
\begin{equation}
  \label{eq:rho-def}
  \rho(C)=12\int_0^1\!\int_0^1 C(u,v)\de u\de v-3,
\end{equation}
see \cite[Sec.~5.1.2]{nelsen2006introduction}, by replacing the uniform weight on the unit square with the density $2(1-u)$.
This density is largest for small $u$, that is, for the leading ranks of the first variable.
Consequently, if $C^\top(u,v)\coloneqq C(v,u)$ denotes the transpose of $C$, then $\nu(C)$ generally differs from $\nu(C^\top)$, and $\nu$ is not a measure of concordance in the sense of Scarsini \cite{scarsini1984measures}.
Such a difference requires $C$ to be non-exchangeable, that is, $C\neq C^\top$, a property that Durante et al.~\cite{durante2010measures} quantify through measures of non-exchangeability.
To remove this dependence on the order of the variables, Genest and Plante \cite[Sec.~4]{genest2003blest} introduced the symmetrized coefficient
\begin{equation}
  \label{eq:eta-def}
  \eta(C)\coloneqq\frac{\nu(C)+\nu(C^\top)}{2}
  =12\int_0^1\!\int_0^1(2-u-v)\,C(u,v)\de u\de v-2,
\end{equation}
which emphasizes the leading ranks of both variables.
It is invariant under transposition but, as observed in \cite[Sec.~4]{genest2003blest}, still violates the reflection axiom of a measure of concordance.
Genest and Plante also derived the limiting distribution of the empirical version of $\nu$ and, through asymptotic efficiency calculations and simulations, assessed it and its symmetric variants as tests of independence.

Question~\ref{q:weighting} is settled at fixed $\rho$: the vertical extent of the $(\rho,\nu)$-region at a given $\rho$ quantifies how much the emphasis on leading ranks can change the coefficient while the overall rank correlation stays the same.
Question~\ref{q:transpose} is settled at fixed $\eta$, which asks how much directional information the symmetrization conceals.
Indeed, $\nu(C^\top)=2\eta(C)-\nu(C)$, so the $(\eta,\nu)$-region is a linear image of the exact region of the pair $(\nu(C),\nu(C^\top))$ and thus answers question~\ref{q:transpose} in full.
In particular, the fibre above $\eta=e$ has length equal to the largest asymmetry $|\nu(C)-\nu(C^\top)|$ among all copulas with $\eta(C)=e$.
Either way, the exact region is a compatibility criterion for two prescribed population values that involves no parametric copula family.

Universal inequalities between rank correlations go back a long way.
Daniels \cite{daniels1950rank} derived inequalities between Spearman's rho and Kendall's tau, and Durbin and Stuart \cite{durbin1951inversions} sharpened this comparison through the inversions in the two rankings, yet the exact region of these two coefficients was found only much later, by Schreyer, Paulin and Trutschnig \cite{schreyer2017exact}, who showed that the classical Durbin--Stuart bound is sharp only at a countable set of points.
A useful inequality, even between familiar coefficients, need not describe the complete range of attainable pairs.
Exact regions have since been determined for many further pairs and triples of measures of association, see, e.g., \cite{kokolbukovsek2021spearman,kokolbukovsek2022exact,kokolbukovsek2023exact,ansari2026rhofootrule}.
For Blest's coefficient itself, Tschimpke \cite{tschimpke2025exact} determined the region with Spearman's rho within the class of bivariate extreme-value copulas, and the exact region with Chatterjee's rank correlation was established in \cite{rockel2026exact}.
To the best of our knowledge, the exact regions of $\nu$ with $\rho$ and with $\eta$ over the class of all copulas have not been determined so far.

Our main results determine both regions, shown in Figure~\ref{fig:regions}.
For the first region, define $\Phi:[-1,1]\to[-1,1]$ by
\begin{equation}
  \label{eq:Phi-def}
  \Phi(r)\coloneqq
  \begin{cases}
    2r+1-\frac34(1+r)^{4/3}, & -1\le r\le0,\\[0.3ex]
    1-\frac34(1-r)^{4/3}, & 0\le r\le1,
  \end{cases}
\end{equation}
and observe that $2r-\Phi(r)=-\Phi(-r)$ for every $r\in[-1,1]$.

\begin{theorem}[Exact $(\rho,\nu)$-region]
\label{thm:rho-nu}
It holds that
\[
  \bigl\{(\rho(C),\nu(C)):C\in\CC\bigr\}
  =\bigl\{(r,n)\in[-1,1]^2:\ 2r-\Phi(r)\le n\le\Phi(r)\bigr\}.
\]
For every $r\in[-1,1]$, exactly one copula $C$ with $\rho(C)=r$ satisfies $\nu(C)=\Phi(r)$, and its survival copula is the only copula with $\rho=r$ and $\nu=2r-\Phi(r)$.
\end{theorem}

For the second region, whose boundary has two regimes, define $\Upsilon:[-1,1]\to[0,\frac{27}{128}]$ by $\Upsilon(-1)\coloneqq0$,
\begin{equation}
  \label{eq:Upsilon-def}
  \Upsilon(e)\coloneqq(1+e)-2^{-1/3}(1+e)^{4/3},
  \qquad e\in\bigl[-\tfrac34,1\bigr],
\end{equation}
and, on $(-1,-\frac34)$, parametrically by
\begin{equation}
  \label{eq:Upsilon-param}
  \begin{gathered}
  \Upsilon(e_b)\coloneqq\frac{b^3(6b^2-15b+8)}{8(1-b)^2},
  \qquad b\in\bigl(0,\tfrac12\bigr),\\[0.5ex]
  \text{where}\qquad
  e_b\coloneqq-1+\frac{b^3(2b^2-9b+8)}{8(1-b)^2}.
  \end{gathered}
\end{equation}
The parametrization is well defined: $b\mapsto e_b$ increases bijectively from $(0,\frac12)$ onto $(-1,-\frac34)$ by Lemma~\ref{lem:eta-nu-values}, and at $e=-\frac34$ both formulas give the value $\frac18$, so $\Upsilon$ is continuous there.

\begin{theorem}[Exact $(\eta,\nu)$-region]
\label{thm:eta-nu}
It holds that
\[
  \bigl\{(\eta(C),\nu(C)):C\in\CC\bigr\}
  =\bigl\{(e,n):\ -1\le e\le1,\ |n-e|\le\Upsilon(e)\bigr\}.
\]
For every $e\in[-1,1]$, exactly one copula $C$ with $\eta(C)=e$ satisfies $\nu(C)=e+\Upsilon(e)$, and its transpose is the only copula with $\eta=e$ and $\nu=e-\Upsilon(e)$.
\end{theorem}

Theorem~\ref{thm:eta-nu} is stated in the coordinates $(\eta,\nu)$ because transposition reflects every fibre about the diagonal there.
Translated to Blest's coefficient and its transpose, it reads
\[
  \bigl\{(\nu(C),\nu(C^\top)):C\in\CC\bigr\}
  =\bigl\{(n,m)\in[-1,1]^2:\ |n-m|\le2\Upsilon\bigl(\tfrac{n+m}{2}\bigr)\bigr\},
\]
see Corollary~\ref{cor:nu-transpose} and Figure~\ref{fig:nu-transpose}.
Again, at each fixed value of $\nu(C)$, unique copulas attain the largest and the smallest value of $\nu(C^\top)$, and part of the boundary takes a closed form: $\nu(C)=n\ge-\frac78$ implies $\nu(C^\top)\le4(\frac{1+n}{2})^{3/4}-n-2$, with equality for exactly one copula.

\begin{figure}[htbp]
  \centering
  \includegraphics[width=0.85\textwidth]{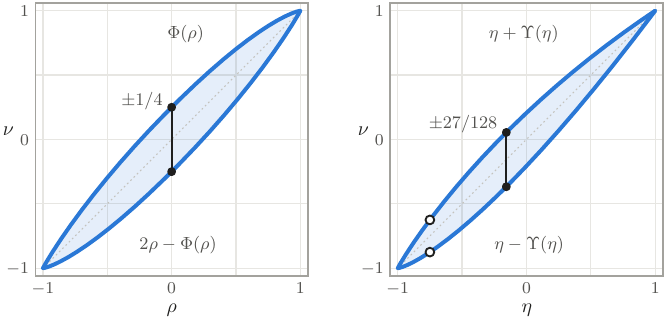}
  \caption{The exact regions of Theorem~\ref{thm:rho-nu} (left) and Theorem~\ref{thm:eta-nu} (right), with the diagonal dotted.
  Each boundary curve is labelled by the function that describes it, and Figures~\ref{fig:rho-extremizers} and~\ref{fig:eta-extremizers} show the copulas attaining it.
  The black segments mark the largest deviations from the diagonal, $|\nu-\rho|=\frac14$ and $|\nu-\eta|=\frac{27}{128}$, and the open circles the change of regime at $\eta=-\frac34$.}
  \label{fig:regions}
\end{figure}

The two theorems yield the sharp inequalities
\[
  |\nu(C)-\rho(C)|\le\frac14
  \qquad\text{and}\qquad
  |\nu(C)-\nu(C^\top)|\le\frac{27}{64},
\]
see Corollaries~\ref{cor:nu-rho} and~\ref{cor:asymmetry}, which also identify the unique extremal copulas.
These inequalities answer questions~\ref{q:weighting} and~\ref{q:transpose}: Blest's coefficient never differs from Spearman's rho by more than $\frac14$, and interchanging the two variables never changes it by more than $\frac{27}{64}$.
The second bound is not a consequence of the first: since $\rho(C^\top)=\rho(C)$ by \eqref{eq:rho-def}, applying the first inequality to $C$ and to $C^\top$ only gives $|\nu(C)-\nu(C^\top)|\le\frac12$.
That the sharp constant is $\frac{27}{64}<\frac12$ means that $\nu(C)-\rho(C)$ and $\nu(C^\top)-\rho(C)$ cannot be extreme in opposite directions at the same time, and determining the constant requires the full $(\eta,\nu)$-region.
Equivalently, by \eqref{eq:eta-def}, $|\nu(C)-\eta(C)|\le\frac{27}{128}$, and a reflection turns this into the further sharp bound $|\rho(C)-\eta(C)|\le\frac{27}{128}$, see Remark~\ref{rem:randomization}~(b).

The two regions call for different arguments, but both rest on the same representation: in the reflected coordinates $(X,Z)=(1-U,1-V)$ of a random vector $(U,V)\sim C$, Blest's coefficient becomes $\nu(C)=12\,\E[X^2Z]-2$, see Lemma~\ref{lem:coupling} in Section~\ref{sec:preliminaries}, so $\nu$, like $\rho$ and $\eta$, is an affine functional of the coupling of two uniform distributions.
At fixed Spearman's rho, the supporting lines of the region amount to ordering a quadratic score, and a rearrangement inequality, together with its equality case, identifies the unique extremizers (Section~\ref{sec:rho-nu}).
At fixed $\eta$, they lead instead to an optimal transport problem whose cost is neither supermodular nor submodular, and we certify two families of extremizers by explicit Kantorovich potentials (Section~\ref{sec:eta-nu}).
This second problem holds a surprise: near countermonotonicity, the upper extremizer is supported on the graph of a function of the second coordinate, but its conditional law given the first coordinate has two atoms on part of the range.
Figures~\ref{fig:rho-extremizers} and~\ref{fig:eta-extremizers} illustrate the supports of the extremizers.

\section{Notation and preliminaries}
\label{sec:preliminaries}

We begin with the representation on which both regions rest and then collect the remaining tools: the behaviour of the coefficients under survival and reflection, a construction of singular copulas from measure-preserving maps that produces every explicit extremizer below, and a rearrangement inequality with equality case.

The representation is formulated in reflected coordinates, in which most computations become more transparent.
Given $(U,V)\sim C$, set $X\coloneqq1-U$ and $Z\coloneqq1-V$.
Both variables are uniform on $[0,1]$, and a large value of $X$ now signals a leading rank of the first variable.
The passage also reverses: every coupling $\gamma$ of two uniform random variables $(X,Z)$ determines one and only one copula in this way.
We may therefore identify $\CC$ with the set $\Gam$ of all such couplings, and we do so whenever convenient.
The class $\CC$ is convex, see \cite[Thm.~1.4.5]{durante2016principles}, so mixtures of copulas are again copulas.
When $(X,Z)\sim\gamma$ for some $\gamma\in\Gam$, we write $\E_\gamma$ for the expectation, and we drop the subscript once the coupling is clear from the context.
Lebesgue measure on $[0,1]$ is denoted by $\mathrm{Leb}$.

\begin{lemma}[Coupling representation]
\label{lem:coupling}
If $(X,Z)$ follows the coupling associated with $C\in\CC$, then
\begin{equation}
  \label{eq:coupling-moments}
  \begin{aligned}
  \rho(C)&=12\,\E[XZ]-3,
  &\nu(C)&=12\,\E[X^2Z]-2,\\
  \nu(C^\top)&=12\,\E[XZ^2]-2,
  &\eta(C)&=6\,\E[XZ(X+Z)]-2.
  \end{aligned}
\end{equation}
Consequently, $\rho$, $\nu$, and $\eta$ depend affinely on the copula.
\end{lemma}

\begin{proof}
The starting point is the identity $C(u,v)=\E[\1\{U\le u\}\1\{V\le v\}]$.
Fubini's theorem then yields
\[
  \int_0^1\!\int_0^1C(u,v)\de u\de v
  =\E[(1-U)(1-V)]=\E[XZ]
\]
and
\[
  \int_0^1\!\int_0^1(1-u)C(u,v)\de u\de v
  =\E\Bigl[\int_U^1(1-u)\de u\int_V^1\de v\Bigr]
  =\frac12\,\E[X^2Z].
\]
Together with \eqref{eq:rho-def} and \eqref{eq:nu-def}, these equations prove the first pair of identities, and the formula for $\nu$ is equivalent to \cite[Frm.~(1)]{genest2003blest}.
The third identity follows because transposition exchanges the coupling of $(X,Z)$ for that of $(Z,X)$.
Averaging the two formulas for $\nu$ and invoking \eqref{eq:eta-def} then yields the representation of $\eta$.
\end{proof}

Next to the transpose $C^\top$ from the introduction, two further transformations of a copula appear throughout, and the next lemma records how the coefficients react to them.
For $(U,V)\sim C$, the survival copula $\widehat C(u,v)\coloneqq u+v-1+C(1-u,1-v)$ is the copula of $(1-U,1-V)$, and the reflection $C^\perp(u,v)\coloneqq u-C(u,1-v)$ in the second argument is the copula of $(U,1-V)$.

\begin{lemma}[Two symmetries]
\label{lem:symmetries}
For every $C\in\CC$,
\begin{equation}
  \label{eq:survival-transform}
  \rho(\widehat C)=\rho(C),\qquad
  \nu(\widehat C)=2\rho(C)-\nu(C),\qquad
  \eta(\widehat C)=2\rho(C)-\eta(C),
\end{equation}
and
\begin{equation}
  \label{eq:reflection-transform}
  \rho(C^\perp)=-\rho(C),\qquad
  \nu(C^\perp)=-\nu(C).
\end{equation}
\end{lemma}

\begin{proof}
For the measure of concordance $\rho$, the stated transformations are standard, see \cite{scarsini1984measures}.
Genest and Plante prove $\nu(C^\perp)=-\nu(C)$ in Section~3 of \cite{genest2003blest}, where they also show that reflection in the first coordinate sends $\nu(C)$ to $\nu(C)-2\rho(C)$, see \cite[Frm.~(5)]{genest2003blest}.
Composing the two reflections gives $\nu(\widehat C)=-\{\nu(C)-2\rho(C)\}$.
For $\eta(\widehat C)$, note that survival commutes with transposition.
Applying the previous identity to $C$ and to $C^\top$ and averaging the results by \eqref{eq:eta-def}, using $\rho(C^\top)=\rho(C)$, establishes the formula.
\end{proof}

Reflection in one coordinate acts on $\eta$ in a different way: rather than a sign change, \eqref{eq:coupling-moments} yields
\begin{equation}
  \label{eq:eta-reflection}
  \eta(C^\perp)=-\rho(C)+\frac{\nu(C^\top)-\nu(C)}{2},
\end{equation}
which can also be deduced from \cite[Sec.~4]{genest2003blest}.
The reflection that negates $(\rho,\nu)$ thus does not negate $(\eta,\nu)$, and in contrast to its $(\rho,\nu)$-counterpart, the $(\eta,\nu)$-region is not symmetric about the origin.

All explicit extremizers below come from a single construction: take a Borel measurable $T:[0,1]\to[0,1]$ that preserves Lebesgue measure, meaning $\mathrm{Leb}(T^{-1}(B))=\mathrm{Leb}(B)$ for every Borel set $B\subseteq[0,1]$, and let $X$ and $Z$ be uniformly distributed on $[0,1]$.
Two couplings result: $\gamma_T\in\Gam$, the law of $(X,T(X))$, and $\gamma^T\in\Gam$, the law of $(T(Z),Z)$, and their copulas are transposes of one another.
Both couplings are singular, as $\gamma_T$ lives on the graph of $T$ read as a function of the first coordinate and $\gamma^T$ on the same graph read as a function of the second.
When a piecewise linear $T$ fails to be injective, the conditional distribution of $Z$ given $X$ under $\gamma^T$ may carry several atoms.
Changing $T$ at finitely many breakpoints does not change its copula, and the support figures show the closures of the linear pieces.

One tool remains: the rearrangement inequality of Hardy and Littlewood, see \cite[Ch.~X]{hardy1952inequalities}, in a probabilistic form that carries the equality case with it.
We know of no convenient reference for that case, so we include the short proof.

\begin{lemma}[Rearrangement inequality with equality case]
\label{lem:rearrangement}
Let $G$ be an integrable random variable with continuous distribution function $F$, and let $Z$ be uniformly distributed on $[0,1]$ and defined on the same probability space.
Then $\E[GZ]\le\E[G\,F(G)]$, with equality if and only if $Z=F(G)$ almost surely.
\end{lemma}

\begin{proof}
Fix $t\in(0,1)$ and set $q_t\coloneqq\sup\{s\in\R:F(s)\le t\}$.
The function $F$ is continuous and nondecreasing, so $F(q_t)=t$ and $\{G>q_t\}=\{F(G)>t\}$, an event of probability $1-t$.
Now take any event $H$ with $\Prob(H)=1-t$.
The indicators $\1_H$ and $\1\{G>q_t\}$ then share the same expectation, and
\[
  \E[G\1_H]-\E\bigl[G\1\{F(G)>t\}\bigr]
  =\E\bigl[(G-q_t)\bigl(\1_H-\1\{G>q_t\}\bigr)\bigr]
  \le0,
\]
because the integrand is nonpositive on $\{G>q_t\}$ and on its complement alike.
As $\Prob(G=q_t)=0$, equality is equivalent to $\1_H=\1\{G>q_t\}$ almost surely.
Now choose $H=\{Z>t\}$ and integrate over $t$.
Since $\E|G|<\infty$, Fubini's theorem applies and gives
\[
  \E[GZ]=\int_0^1\E\bigl[G\1\{Z>t\}\bigr]\de t
  \le\int_0^1\E\bigl[G\1\{F(G)>t\}\bigr]\de t
  =\E[G\,F(G)].
\]
Equality holds precisely when $\1\{Z>t\}=\1\{F(G)>t\}$ almost surely for almost every $t$, in other words when $\E|Z-F(G)|=\E\int_0^1|\1\{Z>t\}-\1\{F(G)>t\}|\de t=0$.
\end{proof}

\section{\texorpdfstring{The exact $(\rho,\nu)$-region}{The exact (rho,nu)-region}}
\label{sec:rho-nu}

To answer question~\ref{q:weighting}, we hold Spearman's rho fixed and ask how far the emphasis on the leading ranks can move Blest's coefficient.
Lemma~\ref{lem:coupling} translates the setting: prescribing $\rho(C)$ prescribes $\E[XZ]$, and $\nu(C)$ is determined by $\E[X^2Z]$, so the two coefficients differ only in how they weight the first coordinate.
The supporting lines of the region thereby turn into rearrangement problems, which Lemma~\ref{lem:rearrangement} solves, equality cases included.
This yields the upper extremizers directly, survival reflection via Lemma~\ref{lem:symmetries} then produces the lower ones, and maximizing the distance between boundary and diagonal gives the sharp deviation $\frac14$.

The moment identities in \eqref{eq:coupling-moments} give, for any $\lambda\in\R$,
\begin{equation}
  \label{eq:rho-nu-support}
  \nu(C)-\lambda\rho(C)=12\,\E\bigl[(X^2-\lambda X)Z\bigr]-2+3\lambda.
\end{equation}
Each supporting-line problem for the $(\rho,\nu)$-region therefore asks for the maximum of $\E[g(X)Z]$ with $g(x)=x^2-\lambda x=(x+c-1)^2-(1-c)^2$ and $c\coloneqq1-\lambda/2$.
Lemma~\ref{lem:rearrangement} points the way: couple $Z$ comonotonically with $(X+c-1)^2$, so that the largest values of $Z$ meet the values of $X$ lying farthest from $1-c$.
For $c\in[0,1]$, the corresponding rank map $\zeta_c:[0,1]\to[0,1]$ reads
\begin{equation}
  \label{eq:Zc-def}
  \begin{aligned}
  \zeta_c(x)&\coloneqq\mathrm{Leb}\bigl(\{y\in[0,1]:|y+c-1|\le|x+c-1|\}\bigr)\\
  &=
  \begin{cases}
    1-x, & x\le1-2c,\\
    2|x+c-1|, & |x+c-1|\le\min\{c,1-c\},\\
    x, & x\ge2-2c.
  \end{cases}
  \end{aligned}
\end{equation}
The three cases cover the whole interval, with the first contributing positive length only for $c<\frac12$ and the last only for $c>\frac12$.
Piecewise linear and Lebesgue measure preserving, the map $\zeta_c$ satisfies $\zeta_0=1-\mathrm{id}$, $\zeta_{1/2}(x)=|2x-1|$, and $\zeta_1=\mathrm{id}$.
Let $D_c$ denote the copula associated with the coupling $\gamma_{\zeta_c}$ from Section~\ref{sec:preliminaries}.
Then $D_0=W$ and $D_1=M$, and for $0<c<1$, the support of $D_c$ in the original coordinates $(u,v)=(1-x,1-z)$ has its kink at the point $(c,1)$.
Three of these copulas appear in Figure~\ref{fig:rho-extremizers}, next to their survival copulas, which attain the lower boundary at the same value of $\rho$.

\begin{figure}[htbp]
  \centering
  \includegraphics[width=0.8\textwidth]{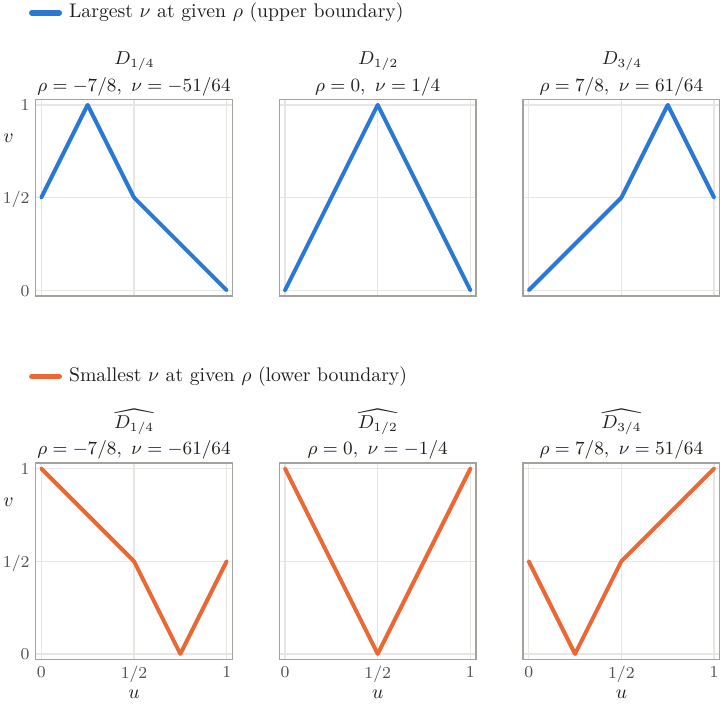}
  \caption{Supports of the extremizers of the $(\rho,\nu)$-region for $c=1/4,1/2,3/4$ from left to right, drawn in the original coordinates $(u,v)=(1-x,1-z)$.
  The top row shows the upper extremizers $D_c$, and the bottom row shows their survival copulas $\widehat{D_c}$, which are the lower extremizers at the same value of $\rho$.
  Each panel states the values of $\rho$ and $\nu$, and the centre column realizes the sharp values $\nu-\rho=\pm1/4$.
  The mass is distributed uniformly along $u$.}
  \label{fig:rho-extremizers}
\end{figure}

\begin{lemma}[Coefficients along the family $D_c$]
\label{lem:rho-nu-values}
For every $c\in[0,1]$,
\begin{equation}
  \label{eq:rho-nu-values}
  \begin{aligned}
  \rho(D_c)&=
  \begin{cases}
    8c^3-1, & c\le\frac12,\\
    1-8(1-c)^3, & c\ge\frac12,
  \end{cases}\\[0.5ex]
  \nu(D_c)&=
  \begin{cases}
    16c^3-12c^4-1, & c\le\frac12,\\
    1-12(1-c)^4, & c\ge\frac12.
  \end{cases}
  \end{aligned}
\end{equation}
Hence $c\mapsto\rho(D_c)$ increases bijectively from $[0,1]$ onto $[-1,1]$, and $\nu(D_c)=\Phi(\rho(D_c))$ for the function $\Phi$ in \eqref{eq:Phi-def}.
\end{lemma}

\begin{proof}
Formula \eqref{eq:rho-nu-values} follows by integrating $x\zeta_c(x)$ and $x^2\zeta_c(x)$ over the linear pieces in \eqref{eq:Zc-def} and applying \eqref{eq:coupling-moments}.
Now write $\rho=\rho(D_c)$ and $\nu=\nu(D_c)$.
For $c\le\frac12$, eliminate $c$ through $(1+\rho)^{4/3}=(8c^3)^{4/3}=16c^4$ to obtain $\nu=2\rho+1-\frac34(1+\rho)^{4/3}$.
For $c\ge\frac12$, the identity $(1-\rho)^{4/3}=16(1-c)^4$ instead leads to $\nu=1-\frac34(1-\rho)^{4/3}$.
\end{proof}

\begin{proof}[Proof of Theorem~\ref{thm:rho-nu}]
Fix $C\in\CC$ with $\rho(C)=r\in[-1,1]$, and take the unique $c\in[0,1]$ with $\rho(D_c)=r$ from Lemma~\ref{lem:rho-nu-values}.
Consider $G\coloneqq(X+c-1)^2$.
This random variable is bounded, and its distribution function $F$ is continuous, since each level set of $x\mapsto(x+c-1)^2$ contains at most two points.
Equation \eqref{eq:Zc-def} also gives $F(G)=\zeta_c(X)$.
Lemma~\ref{lem:rearrangement}, applied to the coupling of $C$, therefore yields
\[
  \E\bigl[(X+c-1)^2Z\bigr]\le\E\bigl[(X+c-1)^2\zeta_c(X)\bigr]=\E_{\gamma_{\zeta_c}}\bigl[(X+c-1)^2Z\bigr],
\]
and equality holds if and only if $Z=\zeta_c(X)$ almost surely, which is to say if and only if $C=D_c$.
Now $x^2-2(1-c)x=(x+c-1)^2-(1-c)^2$, and $\E[Z]=\frac12$ for every coupling, so \eqref{eq:rho-nu-support} with $\lambda=2(1-c)$ shows that $\nu(C)-2(1-c)\rho(C)$ and $12\,\E[(X+c-1)^2Z]$ differ by a constant that does not depend on $C$.
Because $\rho(C)=\rho(D_c)=r$, it follows that $\nu(C)\le\nu(D_c)=\Phi(r)$ by Lemma~\ref{lem:rho-nu-values}, with equality precisely when $C=D_c$.

The survival copula $\widehat C$ satisfies $\rho(\widehat C)=r$ as well, by \eqref{eq:survival-transform}.
Feeding $\widehat C$ into the upper bound and using \eqref{eq:survival-transform} once more gives $\nu(C)=2r-\nu(\widehat C)\ge2r-\Phi(r)$, with equality if and only if $\widehat C=D_c$.
The map $C\mapsto\widehat C$ is an involution of $\CC$, so the lower bound is attained by the survival copula $\widehat{D_c}$ and by no other copula.
It remains to fill the fibre: every mixture $(1-t)D_c+t\widehat{D_c}$ with $t\in[0,1]$ has Spearman's rho $r$ by \eqref{eq:survival-transform}, and as $\nu$ is affine by Lemma~\ref{lem:coupling}, these mixtures run through every value in $[2r-\Phi(r),\Phi(r)]$.
\end{proof}

\begin{corollary}[Sharp deviation between $\nu$ and $\rho$]
\label{cor:nu-rho}
Every $C\in\CC$ satisfies $|\nu(C)-\rho(C)|\le\frac14$.
The maximum $\nu-\rho=\frac14$ is attained only by $D_{1/2}$, which has $\rho=0$ and is supported on the graph of the tent map $u\mapsto\min\{2u,2-2u\}$.
The minimum $\nu-\rho=-\frac14$ is attained only by its survival copula, which is supported on the graph of $u\mapsto|2u-1|$.
\end{corollary}

\begin{proof}
On $[-1,0]$, \eqref{eq:Phi-def} gives $\Phi(r)-r=r+1-\frac34(1+r)^{4/3}$, and its derivative $1-(1+r)^{1/3}$ is positive for $r<0$.
On $[0,1]$, $\Phi(r)-r=1-r-\frac34(1-r)^{4/3}$, and its derivative $(1-r)^{1/3}-1$ is negative for $r>0$.
Consequently $\Phi(r)-r\le\Phi(0)=\frac14$, with equality only at $r=0$.
Lemma~\ref{lem:rho-nu-values} matches $r=0$ with $c=\frac12$, so by Theorem~\ref{thm:rho-nu}, no copula other than $D_{1/2}$ attains the maximum.
Its coupling $\gamma_{\zeta_{1/2}}$ sits on the graph of $\zeta_{1/2}(x)=|2x-1|$ from \eqref{eq:Zc-def}, which becomes the graph of the tent map in the coordinates $(u,v)=(1-x,1-z)$.
For the minimum, \eqref{eq:survival-transform} sends $\nu-\rho$ to $\rho-\nu$, so only $\widehat{D_{1/2}}$ attains it: this is the copula of $(X,\zeta_{1/2}(X))$, supported on the graph of $u\mapsto|2u-1|$.
\end{proof}

The region carries two symmetries worth recording.
First, the map $C\mapsto C^\perp$ sends $(\rho,\nu)$ to $(-\rho,-\nu)$ by \eqref{eq:reflection-transform}, so the region is symmetric about the origin, and the identity $2r-\Phi(r)=-\Phi(-r)$ encodes this symmetry.
Second, by \eqref{eq:survival-transform}, survival leaves $\rho$ unchanged while sending $\nu$ to $2\rho-\nu$, which makes each fibre symmetric about the diagonal.

\section{\texorpdfstring{The exact $(\eta,\nu)$-region}{The exact (eta,nu)-region}}
\label{sec:eta-nu}

We now turn to question~\ref{q:transpose}, which concerns the directional information that symmetrization discards.
At a fixed value of $\eta$, the length of the fibre equals the largest possible asymmetry $|\nu(C)-\nu(C^\top)|$, and two extremal families are needed to attain it: for $\eta\in[-\frac34,1]$, the upper extremizer is supported on a graph over the first coordinate, whereas for $\eta\in(-1,-\frac34)$ its conditional law given the first coordinate carries two atoms on part of the range.
Section~\ref{sec:eta-families} constructs both families and computes their coefficients, Section~\ref{sec:eta-dual} certifies their optimality and uniqueness through explicit transport potentials, and Section~\ref{sec:eta-proof} derives the region, the sharp asymmetry bound, and the exact region of the pair $(\nu(C),\nu(C^\top))$.

By \eqref{eq:eta-def}, transposition sends $(\eta,\nu)$ to $(\eta,2\eta-\nu)$, so every fibre of the $(\eta,\nu)$-region is symmetric about the diagonal, and the quantity $\nu(C)-\eta(C)=\{\nu(C)-\nu(C^\top)\}/2$ measures the asymmetry of Blest's coefficient at $C$.
To describe the supporting lines of the region, fix $\kappa>0$ and define
\[
  d_\kappa(x,z)\coloneqq x^2z-\kappa\,xz^2.
\]
Then \eqref{eq:eta-def} and Lemma~\ref{lem:coupling} give
\begin{equation}
  \label{eq:eta-nu-support}
  \begin{aligned}
  (1+\kappa)\,\nu(C)-2\kappa\,\eta(C)
  &=\nu(C)-\kappa\,\nu(C^\top)\\
  &=12\,\E\bigl[d_\kappa(X,Z)\bigr]-2(1-\kappa).
  \end{aligned}
\end{equation}
The supporting lines of the region thus come from maximizing $\E[d_\kappa(X,Z)]$ over $\Gam$, an optimal transport problem for the cost $-d_\kappa$.
The same reduction recently determined the exact region of Spearman's rho and Spearman's footrule, and \cite{ansari2026rhofootrule} gives a more detailed account of this approach and of its dual certificates.
One feature separates the present problem from Section~\ref{sec:rho-nu}: the function $d_\kappa$ is quadratic in each coordinate, and its mixed derivative $\partial_x\partial_zd_\kappa(x,z)=2(x-\kappa z)$ changes sign along the ray $x=\kappa z$, so $d_\kappa$ is neither supermodular nor submodular and no single rearrangement settles the problem.
Instead, we construct Kantorovich potentials, see \cite[Ch.~5]{villani2009optimal}: functions $\varphi,\psi$ with $\varphi(x)+\psi(z)\ge d_\kappa(x,z)$ everywhere and equality on the candidate support.
Their contact sets pin down the optimizers uniquely, even in the regime where the second coordinate is not a function of the first, see Lemma~\ref{lem:dual-certificate}.

\subsection{Extremal families}
\label{sec:eta-families}

The boundary is attained by two copula families, the first of which is indexed by $w\in[0,1]$: let $A_w$ be the copula associated with $\gamma_{T_w}$, where
\begin{equation}
  \label{eq:Tw-def}
  T_w(x)\coloneqq
  \begin{cases}
    1-x, & 0\le x\le1-w,\\
    x+w-1, & 1-w<x\le1.
  \end{cases}
\end{equation}
Under $\gamma_{T_w}$, the leading fraction $w$ of the first variable moves comonotonically with the trailing fraction of the second, and the first variable's remaining bottom ranks are countermonotone with the second variable's top ranks.
At the endpoints, $A_0=W$ and $A_1=M$.

The second family is driven by a parameter $b\in(0,\frac12]$: set $z_b\coloneqq\frac{b}{2(1-b)}\in(0,\frac12]$ and let $B_b$ be the copula associated with $\gamma^{R_b}$, where
\begin{equation*}
    R_b(z)\coloneqq
  \begin{cases}
    2(1-b)z+1-b, & 0\le z\le z_b,\\[0.3ex]
    \dfrac{(1-b)(1-2z)}{1-2b}, & z_b<z\le b,\\[1ex]
    1-z, & b<z\le1.
  \end{cases}
\end{equation*}
The first branch climbs from $1-b$ to $1$, the second descends from $1$ back to $1-b$, and the final branch covers $[0,1-b)$.
Over $(1-b,1]$, the densities contributed by the first two branches, $\frac{1}{2(1-b)}$ and $\frac{1-2b}{2(1-b)}$, sum to one, so $R_b$ preserves Lebesgue measure.
Under $\gamma^{R_b}$, the bottom fraction $1-b$ of the first variable is countermonotone with the top fraction of the second.
Given a rank $X=x\in(1-b,1]$ in the leading fraction $b$ of the first variable, the conditional law of $Z$ carries two atoms:
\begin{equation}
  \label{eq:Ba-branches}
  \begin{aligned}
  z_-(x)&\coloneqq\frac{x+b-1}{2(1-b)},
  &\quad\text{with probability }&\frac{1}{2(1-b)},\\[0.4ex]
  z_+(x)&\coloneqq\frac{1-b-(1-2b)x}{2(1-b)},
  &\quad\text{with probability }&\frac{1-2b}{2(1-b)}.
  \end{aligned}
\end{equation}
The two atoms are exactly the inverses of the first two pieces of $R_b$.
At $b=\frac12$, the second branch carries no mass and $B_{1/2}=A_{1/2}$.
In both families, the parameter is thus the leading fraction of the first variable that escapes the countermonotone pairing, and $\eta$ increases with it by Lemma~\ref{lem:eta-nu-values} below.
Figure~\ref{fig:eta-extremizers} displays the two regimes, the endpoint they share, and the transposed copulas on the lower boundary.

\begin{figure}[htbp]
  \centering
  \includegraphics[width=0.8\textwidth]{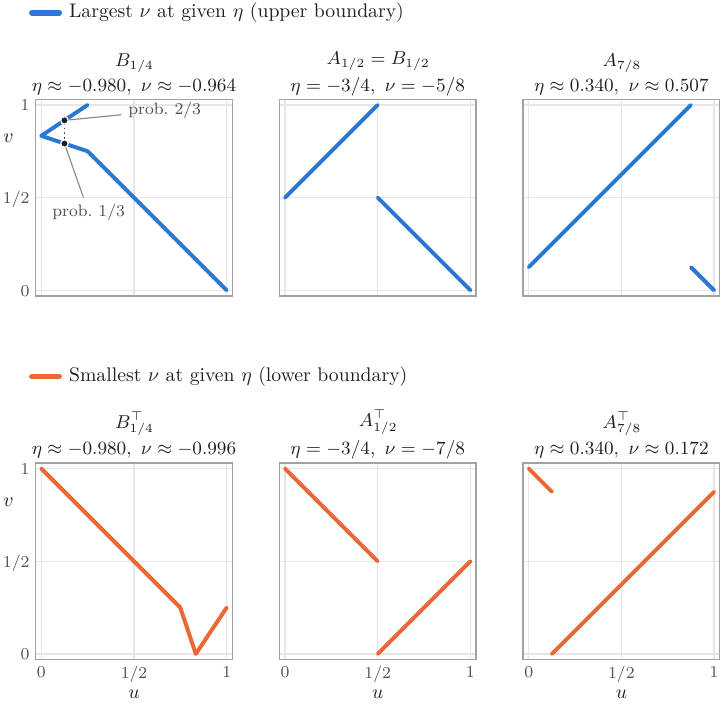}
  \caption{Supports of the extremizers of the $(\eta,\nu)$-region in the original coordinates $(u,v)$.
  The top row shows the upper extremizers $B_{1/4}$, $A_{1/2}=B_{1/2}$, and $A_{7/8}$, and the bottom row shows their transposes, which attain the lower boundary at the same value of $\eta$.
  Each panel states the values of $\eta$ and $\nu$.
  The left column illustrates the randomized regime, which occurs only for $\eta<-3/4$, the centre column shows the regime change at $\eta=-3/4$, and the right column, with $\eta=87/256$, shows a positively dependent case.
  For $B_{1/4}$, where $\eta=-2257/2304$, the support is a graph over $v$, and for $u<1/4$ its two conditional branches have probabilities $2/3$ and $1/3$, illustrated at $u=1/8$.
  In every panel, the mass is spread uniformly along the coordinate over which the support is a graph.}
  \label{fig:eta-extremizers}
\end{figure}

\begin{lemma}[Coefficients along the two families]
\label{lem:eta-nu-values}
For $w\in[0,1]$ and $b\in(0,\frac12]$,
\begin{align}
  \eta(A_w)&=2w^3-1, &
  \nu(A_w)&=2(2-w)w^3-1,
  \label{eq:Aw-values}\\
  \eta(B_b)&=e_b=-1+\frac{b^3(2b^2-9b+8)}{8(1-b)^2}, &
  \nu(B_b)&=-1+\frac{(2-b)b^3}{1-b}.
  \label{eq:Ba-values}
\end{align}
In particular, $\nu(A_w)-\eta(A_w)=2(1-w)w^3$, whereas $\nu(B_b)-\eta(B_b)=\Upsilon(e_b)$ as specified in \eqref{eq:Upsilon-param}.
Both parameter maps $w\mapsto\eta(A_w)$ and $b\mapsto\eta(B_b)$ are strictly increasing.
For $w\in[\frac12,1]$, the first ranges over $[-\frac34,1]$, and for $b\in(0,\frac12]$, the second ranges over $(-1,-\frac34]$.
\end{lemma}

\begin{proof}
Direct integration over the linear pieces of \eqref{eq:Tw-def}, combined with \eqref{eq:coupling-moments}, gives \eqref{eq:Aw-values}.
For \eqref{eq:Ba-values}, weight the two branches as in \eqref{eq:Ba-branches}.
One of the required integrals, for instance, reads
\[
  \E[X^2Z]=\int_0^{1-b}x^2(1-x)\de x+\int_{1-b}^1x^2\,\frac{z_-(x)+(1-2b)\,z_+(x)}{2(1-b)}\de x.
\]
That $\eta(A_w)$ increases in $w$ is plain from its formula, and differentiating \eqref{eq:Ba-values} gives
\[
  \partial_b\eta(B_b)=\frac{b^2(2-b)(3b^2-8b+6)}{4(1-b)^3}>0.
\]
In addition, $\eta(B_{1/2})=-\frac34$, and $\eta(B_b)\to-1$ as $b\to0$.
\end{proof}

\subsection{Dual certificates}
\label{sec:eta-dual}

Both regimes rest on the following principle, which we state for a general continuous cost $c$.
If two potentials $\varphi,\psi$ satisfy $\varphi(x)+\psi(z)\ge c(x,z)$ everywhere, then the sum of the integrals of $\varphi$ and $\psi$ over $[0,1]$ bounds the expected cost of every coupling in $\Gam$, and a coupling concentrated on the contact set, where equality holds, attains this bound.
The potentials thus certify that such a coupling is optimal, and if the contact set is a graph up to finitely many lines, they also certify that it is the only optimal coupling.

\begin{lemma}[Dual certificates]
\label{lem:dual-certificate}
Let $c:[0,1]^2\to\R$ and $\varphi,\psi:[0,1]\to\R$ be continuous functions such that
\[
  \Delta(x,z)\coloneqq\varphi(x)+\psi(z)-c(x,z)\ge0
  \qquad\text{for all }(x,z)\in[0,1]^2,
\]
and let $K\coloneqq\{(x,z)\in[0,1]^2:\Delta(x,z)=0\}$ be their contact set.
\begin{enumerate}[label=(\alph*),leftmargin=2em]
\item
Every $\gamma\in\Gam$ satisfies
\[
  \E_\gamma[c(X,Z)]\le\int_0^1\varphi(t)\de t+\int_0^1\psi(t)\de t,
\]
with equality if and only if $\gamma(K)=1$.
In particular, if some $\gamma^*\in\Gam$ satisfies $\gamma^*(K)=1$, then the maximizers of $\gamma\mapsto\E_\gamma[c(X,Z)]$ over $\Gam$ are exactly the couplings concentrated on $K$.
\item
Suppose that $K\subseteq\{(x,T(x)):x\in[0,1]\}\cup(N_1\times[0,1])\cup([0,1]\times N_2)$ for a Borel map $T:[0,1]\to[0,1]$ and finite sets $N_1,N_2\subseteq[0,1]$.
Then every $\gamma\in\Gam$ with $\gamma(K)=1$ is the law $\gamma_T$ of $(X,T(X))$, and in particular $T$ preserves Lebesgue measure.
The same holds with the roles of the two coordinates exchanged and with $\gamma^T$ in place of $\gamma_T$.
\end{enumerate}
\end{lemma}

\begin{proof}
Part~(a) is the elementary direction of Kantorovich duality, see \cite[Thm.~5.10]{villani2009optimal}, applied to the cost $-c$: integrating $\Delta\ge0$ against $\gamma$ and using the uniform marginals gives the bound, and equality holds exactly when $\Delta=0$ holds $\gamma$-almost surely.
For~(b), the sets $N_1\times[0,1]$ and $[0,1]\times N_2$ are $\gamma$-null because the marginals have no atoms, so $\gamma$ is concentrated on the graph of $T$ and is therefore the law of $(X,T(X))$.
\end{proof}

It remains to construct such potentials for the costs $d_\kappa$, and we start with the regime in which the optimal coupling is supported on a graph over the first coordinate.
For $w\in[\frac12,1)$, set $\kappa_w\coloneqq\frac{3-2w}{2w}$, and define continuous functions $\varphi_w,\psi_w:[0,1]\to\R$ by $\varphi_w(1-w)=0$, $\psi_w(0)=0$, and
\begin{equation}
  \label{eq:potential-A}
  \begin{gathered}
  \varphi_w'(x)=
  \begin{cases}
    (1-x)\bigl\{2x-\kappa_w(1-x)\bigr\}, & 0<x<1-w,\\
    (x+w-1)\bigl\{2x-\kappa_w(x+w-1)\bigr\}, & 1-w<x<1,
  \end{cases}\\[0.5ex]
  \psi_w'(z)=
  \begin{cases}
    (z+1-w)\bigl\{z+1-w-2\kappa_wz\bigr\}, & 0<z<w,\\
    (1-z)\bigl\{1-z-2\kappa_wz\bigr\}, & w<z<1.
  \end{cases}
  \end{gathered}
\end{equation}
These definitions align the derivatives of $\varphi_w$ and $\psi_w$ with the corresponding partial derivatives of $d_{\kappa_w}$ along both branches of $\gamma_{T_w}$.
The support of $\gamma_{T_w}$ is
\[
  S_w\coloneqq\{(x,1-x):0\le x\le1-w\}
       \cup\{(x,x+w-1):1-w\le x\le1\}.
\]

\begin{proposition}[Potentials for the graph regime]
\label{prop:dual-A}
Let $w\in[\frac12,1)$.
Then every $(x,z)\in[0,1]^2$ satisfies
\begin{equation}
  \label{eq:dual-A-inequality}
  \varphi_w(x)+\psi_w(z)\ge d_{\kappa_w}(x,z).
\end{equation}
If $w>\frac12$, equality holds exactly on $S_w$, and if $w=\frac12$, it holds exactly on $S_{1/2}\cup([\frac12,1]\times\{\frac12\})$.
\end{proposition}

\begin{proof}
Set $\Delta(x,z)\coloneqq\varphi_w(x)+\psi_w(z)-d_{\kappa_w}(x,z)$.
By \eqref{eq:potential-A}, $\frac{\mathrm d}{\mathrm dx}\Delta(x,x+w-1)=0$ on $(1-w,1)$ and $\frac{\mathrm d}{\mathrm dx}\Delta(x,1-x)=0$ on $(0,1-w)$.
In addition, $\Delta(1-w,0)=0$, and evaluating the integral that defines $\psi_w$ yields $\psi_w(w)=d_{\kappa_w}(1-w,w)$.
Taken together, these facts force $\Delta$ to vanish everywhere on $S_w$.

Nonnegativity away from that set is proved rectangle by rectangle, after dividing the square into four of them.
Each of the following factorizations is an elementary polynomial identity, and every step uses $w\ge\frac12$.
The first rectangle is $[1-w,1]\times[0,w]$, where
\[
  \begin{aligned}
  \Delta(x,z)&=\frac{(x+w-1-z)^2\,\bigl\{1-w+(2w-1)x-2(1-w)z\bigr\}}{2w},\\
  1-w+(2w-1)x-2(1-w)z&\ge2(1-w)(w-z)\ge0.
  \end{aligned}
\]
For $w>\frac12$, $\Delta$ vanishes on this rectangle exactly on the line $z=x+w-1$ and at the corner $(1-w,w)$.
For $w=\frac12$, the second factor also vanishes along $z=\frac12$, and this produces the additional horizontal segment.

Next comes $[0,1-w]\times[w,1]$, where
\[
  \begin{aligned}
  \Delta(x,z)&=\frac{(x+z-1)^2\,F(x,z)}{6w},\\
  F(x,z)&\coloneqq(6-2w)z-(3+2w)x+3-4w.
  \end{aligned}
\]
The function $F$ decreases in $x$ and increases in $z$, so $F\ge F(1-w,w)=3w>0$.
Equality thus occurs along $z=1-x$ and nowhere else.

The third rectangle is $[0,1-w]\times[0,w]$, on which
\[
  \partial_x\Delta(x,z)=\frac{(x+z-1)\,\bigl\{(3-2w)z-(3+2w)x+3-2w\bigr\}}{2w}\le0,
\]
because $x+z\le1$ and because the second factor decreases in $x$ and is nonnegative at $x=1-w$, where it equals $w(2w-1)+(3-2w)z$.
Hence $\Delta(x,z)\ge\Delta(1-w,z)$.
Along the edge $x=1-w$,
\[
  \partial_z\Delta(1-w,z)=-\frac{(1-w)z(3z-2w)}{w}.
\]
Up to $z=\frac{2w}{3}$ this derivative is nonnegative, and thereafter it is nonpositive.
Consequently $\Delta(1-w,\cdot)\ge\min\{\Delta(1-w,0),\Delta(1-w,w)\}=0$.

Exchanging the roles of the coordinates handles the fourth rectangle $[1-w,1]\times[w,1]$:
\begin{gather*}
  \partial_z\Delta(x,z)=\frac{(x+z-1)\,\bigl\{(3-w)z-wx-w\bigr\}}{w}\ge0
  \qquad\text{and}\\
  \partial_x\Delta(x,w)=\frac{(1-x)(2w-1)(3-2w-3x)}{2w}.
\end{gather*}
Here the inequality holds since $x+z\ge1$ and since the second factor increases in $z$ and equals $w(2-w-x)\ge0$ at $z=w$.
It follows that $\Delta(x,z)\ge\Delta(x,w)$.
Along this edge, the derivative changes from positive to negative if $w>\frac12$ and is identically zero if $w=\frac12$.
In either case the minimum on the edge sits at an endpoint, and both endpoint values vanish.

On the final two rectangles, equality for $w>\frac12$ is limited to the three points $(1-w,0)$, $(1-w,w)$, and $(1,w)$, all of which belong to $S_w$.
When $w=\frac12$, the edge $[\frac12,1]\times\{\frac12\}$ also has zero slack, and the displayed derivative inequalities rule out any other contact points.
The claimed equality sets follow.
\end{proof}

For the second regime, fix $b\in(0,\frac12)$ and put $\kappa_b\coloneqq\frac{2(1-b)}{b}$.
No confusion with $\kappa_w$, $\varphi_w$, and $\psi_w$ can arise from this notation or from $\varphi_b$ and $\psi_b$ below, because $b<\frac12\le w$.
Recall $z_b=\frac{b}{2(1-b)}$ and the branches $z_\mp$ from \eqref{eq:Ba-branches}, whose inverses are the first two pieces of $R_b$: $R_b^-(z)\coloneqq2(1-b)z+1-b$ on $[0,z_b]$ and $R_b^+(z)\coloneqq\frac{(1-b)(1-2z)}{1-2b}$ on $[z_b,b]$.
Let $\varphi_b,\psi_b:[0,1]\to\R$ be continuous, normalized by $\varphi_b(1-b)=0$ and $\psi_b(0)=0$, with derivatives
\begin{equation}
  \label{eq:potential-B}
  \begin{gathered}
  \varphi_b'(x)=
  \begin{cases}
    (1-x)\bigl\{2x-\kappa_b(1-x)\bigr\}, & 0<x<1-b,\\[0.3ex]
    z_-(x)\bigl\{2x-\kappa_bz_-(x)\bigr\}, & 1-b<x<1,
  \end{cases}\\[0.5ex]
  \psi_b'(z)=
  \begin{cases}
    R_b^-(z)\bigl\{R_b^-(z)-2\kappa_bz\bigr\}, & 0<z<z_b,\\[0.3ex]
    R_b^+(z)\bigl\{R_b^+(z)-2\kappa_bz\bigr\}, & z_b<z<b,\\[0.3ex]
    (1-z)\bigl\{1-z-2\kappa_bz\bigr\}, & b<z<1.
  \end{cases}
  \end{gathered}
\end{equation}
Over $(1-b,1)$, the coupling $\gamma^{R_b}$ runs along two branches, yet $\varphi_b'$ is defined through $z_-$ alone.
No inconsistency arises, because
\begin{equation}
  \label{eq:branch-sum}
  z_-(x)+z_+(x)=\frac{2x}{\kappa_b}
  \quad\Longrightarrow\quad
  \partial_xd_{\kappa_b}(x,z_-(x))=\partial_xd_{\kappa_b}(x,z_+(x)),
\end{equation}
and this equality is exactly the first-order condition that permits an optimal coupling to split its mass between two branches.

\begin{proposition}[Potentials for the randomized regime]
\label{prop:dual-B}
Let $b\in(0,\frac12)$.
Every $(x,z)\in[0,1]^2$ satisfies
\begin{equation*}
    \varphi_b(x)+\psi_b(z)\ge d_{\kappa_b}(x,z),
\end{equation*}
and equality holds exactly when $x=R_b(z)$.
\end{proposition}

\begin{proof}
Put $\Delta(x,z)\coloneqq\varphi_b(x)+\psi_b(z)-d_{\kappa_b}(x,z)$.
Just as in Proposition~\ref{prop:dual-A}, equation \eqref{eq:potential-B} makes the derivative of $\Delta$ vanish along all three branches of $R_b$.
Three anchors remain: the identity $\Delta(1-b,0)=0$ anchors $R_b^-$, continuity of $\psi_b$ at $z_b$ anchors $R_b^+$ where the branches meet at $(1,z_b)$, and evaluating the defining integrals gives $\psi_b(b)=d_{\kappa_b}(1-b,b)$, which anchors the antidiagonal branch.
Hence $\Delta=0$ on the graph of $R_b$.

The square splits into six rectangles along the lines $x=1-b$ and $z\in\{z_b,b\}$.
On $[1-b,1]\times[0,z_b]$ and $[1-b,1]\times[z_b,b]$, respectively, polynomial factorization gives
\[
  \Delta=\frac{\ell_-^2\,F_-}{6b(1-b)}
  \qquad\text{and}\qquad
  \Delta=\frac{\ell_+^2\,F_+}{6b(1-b)(1-2b)^2},
\]
where
\[
  \begin{aligned}
  \ell_-(x,z)&\coloneqq x-2(1-b)z-(1-b),\\
  \ell_+(x,z)&\coloneqq(1-2b)x+2(1-b)z-(1-b),\\
  F_-(x,z)&\coloneqq1-b^2-(1-2b)x-2(1-b)(2-b)z,\\
  F_+(x,z)&\coloneqq(1-b)(1-3b)-(1-2b)x+2(1-b)(2-3b)z.
  \end{aligned}
\]
The graphs of $R_b^-$ and $R_b^+$ are exactly the zero sets of the linear factors $\ell_-$ and $\ell_+$, respectively.
In the variable $z$, $F_-$ decreases and $F_+$ increases, and at $z=z_b$ each of them equals $(1-2b)(1-x)\ge0$.
Both rectangles therefore satisfy $\Delta\ge0$, with equality exactly on the corresponding branches $R_b^\mp$.

On $[0,1-b]\times[b,1]$, the factorization reads
\[
  \Delta(x,z)=\frac{(x+z-1)^2\,\bigl\{(4-3b)z-2x+2-3b\bigr\}}{3b}.
\]
Here the second factor decreases in $x$, increases in $z$, and takes the positive value $3b(1-b)$ at $(1-b,b)$.
Equality therefore occurs on the antidiagonal branch and nowhere else.

On each of the rectangles $[0,1-b]\times[0,z_b]$ and $[0,1-b]\times[z_b,b]$,
\[
  \partial_x\Delta(x,z)=-\frac{2(1-x-z)\,\bigl\{(1-b)(1+z)-x\bigr\}}{b}\le0,
\]
because $x\le1-b$ and $z\le b$.
As a result, $\Delta(x,z)\ge\Delta(1-b,z)$.
Along the edge $x=1-b$,
\begin{gather*}
  \partial_z\Delta(1-b,z)=-\frac{4(1-b)^2z\,\bigl\{(2-b)z-b\bigr\}}{b}
  \quad\text{on }(0,z_b),\\
  \partial_z\Delta(1-b,z)=-\frac{4(1-b)^2(b-z)\,\bigl\{(2-3b)z-b(1-b)\bigr\}}{b(1-2b)^2}
  \quad\text{on }(z_b,b).
\end{gather*}
The first of these expressions changes sign from positive to negative, and the second is nonpositive because $z\ge z_b>\frac{b(1-b)}{2-3b}$.
Combined with $\Delta(1-b,0)=0$, $\Delta(1-b,z_b)=\frac{b^2(1-2b)}{6(1-b)}>0$, and $\Delta(1-b,b)=0$, this proves $\Delta(1-b,\cdot)\ge0$ on $[0,b]$.
On both rectangles, then, $\Delta\ge0$, and equality occurs there only at $(1-b,0)$ and $(1-b,b)$.

On the last rectangle, $[1-b,1]\times[b,1]$,
\begin{gather*}
  \partial_x\Delta(x,z)=-\frac{\ell_-(x,z)\,\ell_+(x,z)}{2b(1-b)}\ge0
  \qquad\text{and}\\
  \partial_z\Delta(1-b,z)=\frac{(z-b)\,\bigl\{(1-b)^2+(4-3b)z-1\bigr\}}{b}\ge0.
\end{gather*}
For the first derivative, note that $\ell_-\le\ell_-(1,b)=-b(1-2b)<0$ and $\ell_+\ge\ell_+(1-b,b)=0$ on this rectangle.
For the second, both factors are nonnegative once $z\ge b$, and the second factor is at least $2b(1-b)$.
Consequently $\Delta(x,z)\ge\Delta(1-b,z)\ge\Delta(1-b,b)=0$, and equality is possible only at $(1-b,b)$.
The proof is complete.
\end{proof}

\subsection{The exact regions and the sharp asymmetry bound}
\label{sec:eta-proof}

\begin{proof}[Proof of Theorem~\ref{thm:eta-nu}]
Take $w\in[\frac12,1)$ and $\gamma\in\Gam$, and let $C$ be the copula associated with $\gamma$.
Proposition~\ref{prop:dual-A} shows that the potentials $\varphi_w,\psi_w$ satisfy the assumption of Lemma~\ref{lem:dual-certificate} for the cost $d_{\kappa_w}$, with the coupling $\gamma_{T_w}$ concentrated on their contact set.
Lemma~\ref{lem:dual-certificate}~(a) therefore yields
\[
  \begin{aligned}
  \E_\gamma[d_{\kappa_w}(X,Z)]
  &\le\int_0^1\varphi_w(t)\de t+\int_0^1\psi_w(t)\de t\\
  &=\E_{\gamma_{T_w}}[d_{\kappa_w}(X,Z)].
  \end{aligned}
\]
Applying \eqref{eq:eta-nu-support} to $C$ and to $A_w$ turns this inequality into
\begin{equation}
  \label{eq:eta-nu-support-line}
  (1+\kappa_w)\,\nu(C)-2\kappa_w\,\eta(C)
  \le(1+\kappa_w)\,\nu(A_w)-2\kappa_w\,\eta(A_w)
  \qquad\text{for every }C\in\CC.
\end{equation}
The corresponding inequality with $(B_b,\kappa_b)$ in place of $(A_w,\kappa_w)$ follows in the same way from Proposition~\ref{prop:dual-B}.

\smallskip\noindent\emph{Upper bound.}\quad
Fix $C\in\CC$ with $\eta(C)=e$.
Suppose first that $e\in[-\frac34,1)$, and choose $w\coloneqq\bigl(\frac{1+e}{2}\bigr)^{1/3}\in[\frac12,1)$, so that $\eta(A_w)=e$ by \eqref{eq:Aw-values}.
Combining the supporting inequality \eqref{eq:eta-nu-support-line} with \eqref{eq:Aw-values} and \eqref{eq:Upsilon-def} gives $\nu(C)\le\nu(A_w)=e+2(1-w)w^3=e+\Upsilon(e)$.

Now let $e\in(-1,-\frac34)$, and take the unique $b\in(0,\frac12)$ with $e_b=e$ from Lemma~\ref{lem:eta-nu-values}.
The same reasoning, run with $B_b$ and with \eqref{eq:Ba-values} and \eqref{eq:Upsilon-param}, gives $\nu(C)\le\nu(B_b)=e+\Upsilon(e)$.

At $e=\pm1$, the Fr\'echet--Hoeffding bounds $W\le C\le M$ \cite[Thm.~2.2.3]{nelsen2006introduction} and formula \eqref{eq:nu-def} imply $-1\le\nu(C),\nu(C^\top)\le1$.
Their average $\eta(C)$, see \eqref{eq:eta-def}, reaches either endpoint only when each term does.
The weight in \eqref{eq:nu-def} is strictly positive in the interior and copulas are continuous, so $\nu(C)=\nu(M)$ forces $C=M$ and $\nu(C)=\nu(W)$ forces $C=W$.
The endpoint fibres are thus singletons, realized uniquely by $M$ and $W$.

\smallskip\noindent\emph{Lower bound and fibres.}\quad
Since $\eta(C^\top)=\eta(C)=e$ by \eqref{eq:eta-def}, applying the upper estimate to $C^\top$ shows that
\[
  \nu(C)=2e-\nu(C^\top)\ge e-\Upsilon(e).
\]
Now fix $e\in[-1,1]$ and let $C^*$ attain the upper bound.
The coefficients $\eta$ and $\nu$ are affine by Lemma~\ref{lem:coupling}, and $\eta(C^{*\top})=\eta(C^*)$ by \eqref{eq:eta-def}, so as $t$ ranges over $[0,1]$, the mixtures $(1-t)C^*+t\,C^{*\top}$ sweep out the full fibre.

\smallskip\noindent\emph{Uniqueness.}\quad
Suppose first that $\nu(C)=e+\Upsilon(e)$ for some $e\in[-\frac34,1)$.
Then \eqref{eq:eta-nu-support-line} holds with equality, the coupling of $C$ attains the bound in Lemma~\ref{lem:dual-certificate}~(a), and it is therefore concentrated on the contact set from Proposition~\ref{prop:dual-A}.
This set is the graph of $T_w$, apart from the point $(1-w,0)$ and, for $w=\frac12$, the segment $[\frac12,1]\times\{\frac12\}$.
Applying Lemma~\ref{lem:dual-certificate}~(b) with $N_1=\{1-w\}$, and with $N_2=\{\frac12\}$ if $w=\frac12$ and $N_2=\emptyset$ otherwise, gives $C=A_w$.

For $e\in(-1,-\frac34)$, the same argument through Proposition~\ref{prop:dual-B} concentrates the coupling of $C$ on $\{(R_b(z),z):z\in[0,1]\}$, a graph over the second coordinate, so the last part of Lemma~\ref{lem:dual-certificate}~(b) gives $C=B_b$.
Transposition, finally, is an involution of $\CC$, so uniqueness at the upper boundary carries over to the lower one.
\end{proof}

\begin{corollary}[Maximal asymmetry of Blest's coefficient]
\label{cor:asymmetry}
Every copula $C\in\CC$ obeys
\[
  |\nu(C)-\eta(C)|\le\frac{27}{128},
  \qquad\text{equivalently}\qquad
  \bigl|\nu(C)-\nu(C^\top)\bigr|\le\frac{27}{64}.
\]
Both constants are sharp: the maximum $\frac{27}{128}$ of $\nu-\eta$ is attained only by $A_{3/4}$, at which $(\eta,\nu)=(-\frac{5}{32},\frac{7}{128})$, and the minimum $-\frac{27}{128}$ only by $A_{3/4}^\top$.
\end{corollary}

\begin{proof}
Differentiating \eqref{eq:Upsilon-def} on $[-\frac34,1]$ gives $1-\frac43\,2^{-1/3}(1+e)^{1/3}$.
On that interval, $\Upsilon$ is therefore strictly concave, with its maximum where $1+e=\frac{27}{32}$, that is, at $e=-\frac{5}{32}$.
There, $\Upsilon(-\frac{5}{32})=\frac{27}{32}(1-\frac34)=\frac{27}{128}$, and $\eta(A_w)=-\frac{5}{32}$ for $w=\frac34$ by \eqref{eq:Aw-values}.

Along the parametric portion, differentiating \eqref{eq:Upsilon-param} gives
\[
  \partial_b\Upsilon(e_b)
  =\frac{b^2(2-3b)(3b^2-8b+6)}{4(1-b)^3}>0,
  \qquad b\in(0,\tfrac12).
\]
On that part, consequently, $\Upsilon<\Upsilon(-\frac34)=\frac18<\frac{27}{128}$, as $e_b$ increases in $b$ by Lemma~\ref{lem:eta-nu-values}.
The uniqueness assertion in Theorem~\ref{thm:eta-nu} and \eqref{eq:Aw-values} now deliver the equality cases.
The identity $\nu(C)-\eta(C)=\{\nu(C)-\nu(C^\top)\}/2$, a consequence of \eqref{eq:eta-def}, turns all of this into the equivalent asymmetry bound.
\end{proof}

How Blest's coefficient reacts when the two variables are interchanged is visible most directly in the coordinates $(\nu(C),\nu(C^\top))$, where the boundary of the region, shown in Figure~\ref{fig:nu-transpose}, becomes the graph of the following function together with its inverse.
Define $\Lambda:[-1,1]\to[-1,1]$ by $\Lambda(-1)\coloneqq-1$,
\begin{equation}
  \label{eq:Lambda-def}
  \Lambda(n)\coloneqq4\Bigl(\frac{1+n}{2}\Bigr)^{3/4}-n-2,
  \qquad n\in\bigl[-\tfrac78,1\bigr],
\end{equation}
and, on $(-1,-\frac78)$, parametrically by
\begin{equation}
  \label{eq:Lambda-param}
  \begin{gathered}
  \Lambda(n_b)\coloneqq-1+\frac{(2-b)b^3}{1-b},
  \qquad b\in\bigl(0,\tfrac12\bigr),\\[0.5ex]
  \text{where}\qquad
  n_b\coloneqq-1+\frac{b^4(3-2b)}{4(1-b)^2}.
  \end{gathered}
\end{equation}
At $-\frac78$, both formulas give the value $-\frac58$.
The proof below shows that $b\mapsto n_b$ increases bijectively from $(0,\frac12)$ onto $(-1,-\frac78)$ and that $\Lambda$ is a strictly increasing bijection of $[-1,1]$.

\begin{corollary}[Exact region of Blest's coefficient and its transpose]
\label{cor:nu-transpose}
It holds that
\begin{align*}
  \bigl\{(\nu(C),\nu(C^\top)):C\in\CC\bigr\}
  &=\bigl\{(n,m)\in[-1,1]^2:\ |n-m|\le2\Upsilon\bigl(\tfrac{n+m}{2}\bigr)\bigr\}\\
  &=\bigl\{(n,m)\in[-1,1]^2:\ \Lambda^{-1}(n)\le m\le\Lambda(n)\bigr\}.
\end{align*}
For every $n\in[-1,1]$, the maximum $\Lambda(n)$ and the minimum $\Lambda^{-1}(n)$ of $\nu(C^\top)$ over all copulas $C$ with $\nu(C)=n$ are each attained by exactly one copula.
For $n=-1$, this copula is $W$.
For $n\in(-1,1]$, the maximizer is the transpose of an upper extremizer of Theorem~\ref{thm:eta-nu}, namely $A_w^\top$ with $w=(\frac{1+n}{2})^{1/4}$ if $n\ge-\frac78$ and $B_b^\top$ with $n_b=n$ if $n<-\frac78$, whereas the minimizer is an upper extremizer $A_w$ or $B_b$ itself.
\end{corollary}

\begin{figure}[htbp]
  \centering
  \includegraphics[width=0.85\textwidth]{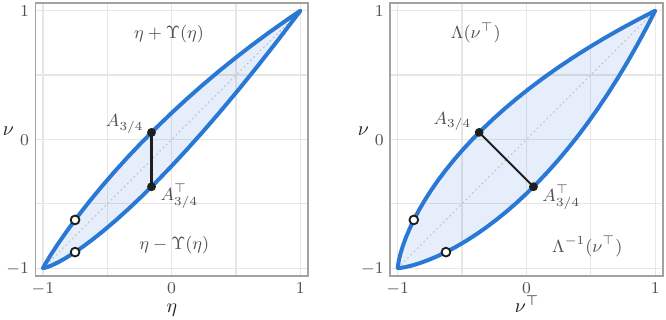}
  \caption{Keeping $\nu$ and replacing $\eta$ by $\nu^\top=2\eta-\nu$ carries the $(\eta,\nu)$-region of Theorem~\ref{thm:eta-nu} (left) onto the region of all pairs $(\nu^\top,\nu)=(\nu(C^\top),\nu(C))$ (right), which is symmetric about the dotted diagonal and hence coincides with the region of Corollary~\ref{cor:nu-transpose}.
  Each point moves horizontally, so the upper extremizers $A_w$ and $B_b$ stay on the upper boundary, which becomes the graph of $\Lambda$, and their transposes stay on the lower one, the graph of $\Lambda^{-1}$.
  The black segments join $A_{3/4}$ and $A_{3/4}^\top$, which realize the largest deviations $|\nu-\eta|=\frac{27}{128}$ and $|\nu-\nu^\top|=\frac{27}{64}$, and the open circles mark the change of regime at $\eta=-\frac34$ and its image.}
  \label{fig:nu-transpose}
\end{figure}

\begin{proof}
Consider the linear bijection $L(e,n)\coloneqq(n,2e-n)$ of $\R^2$.
Since $\nu(C^\top)=2\eta(C)-\nu(C)$ by \eqref{eq:eta-def}, it maps $(\eta(C),\nu(C))$ to $(\nu(C),\nu(C^\top))$ for every $C\in\CC$, and a copula attains a point of the $(\eta,\nu)$-region exactly when it attains the image of this point under $L$.
The first set equality then follows from Theorem~\ref{thm:eta-nu}, because $L^{-1}(n,m)=(\frac{n+m}{2},n)$ and $|n-\frac{n+m}{2}|=\frac{|n-m|}{2}$.

Under $L$, the lower and upper boundary curves $e\mapsto(e,e\mp\Upsilon(e))$ of the $(\eta,\nu)$-region land on
\[
  K_\pm\coloneqq\bigl\{\bigl(e\mp\Upsilon(e),\,e\pm\Upsilon(e)\bigr):e\in[-1,1]\bigr\},
\]
respectively.
Both coordinates of these curves increase strictly in $e$, because $|\Upsilon'|<1$ on $(-1,1)$.
To see the latter, differentiate \eqref{eq:Upsilon-def} to get $\Upsilon'(e)=1-\frac43\,2^{-1/3}(1+e)^{1/3}\in[-\frac13,\frac13]$ for $e\in[-\frac34,1)$, and divide the derivatives of $\Upsilon(e_b)$ and $e_b$ with respect to $b$, computed in the proofs of Corollary~\ref{cor:asymmetry} and Lemma~\ref{lem:eta-nu-values}, to get $\Upsilon'(e_b)=\frac{2-3b}{2-b}\in(\frac13,1)$ for $b\in(0,\frac12)$.
So $K_+$ is the graph of a strictly increasing bijection of $[-1,1]$, and its reflection $K_-$ in the diagonal is the graph of the inverse bijection.
As the image of the convex set $\CC$ under a map that is affine by Lemma~\ref{lem:coupling}, the $(\nu,\nu^\top)$-region is convex, it contains the diagonal points $L(e,e)=(e,e)$ by Theorem~\ref{thm:eta-nu}, and its boundary is $K_+\cup K_-$.
The fibre above $n$ is therefore the interval between the points of $K_-$ and $K_+$ with first coordinate $n$.

What remains is to identify $K_+$ with the graph of $\Lambda$.
At parameter $e$, the point of $K_+$ is $L(e,e-\Upsilon(e))$, and by Theorem~\ref{thm:eta-nu} only the transpose of the upper extremizer at $e$ attains it.
For $e\in[-\frac34,1]$ and $w=(\frac{1+e}{2})^{1/3}$, \eqref{eq:eta-def} and \eqref{eq:Aw-values} give
\[
  \nu(A_w^\top)=2\eta(A_w)-\nu(A_w)=2w^4-1,
  \qquad
  \nu(A_w)=4w^3-2w^4-1.
\]
Set $n\coloneqq\nu(A_w^\top)\in[-\frac78,1]$ and eliminate $w$ through $w^4=\frac{1+n}{2}$: since $\nu((A_w^\top)^\top)=\nu(A_w)$, this yields \eqref{eq:Lambda-def}.
For $e=e_b$ with $b\in(0,\frac12)$, \eqref{eq:Ba-values} likewise gives $\nu(B_b^\top)=2e_b-\nu(B_b)=n_b$ and $\nu(B_b)=\Lambda(n_b)$ as in \eqref{eq:Lambda-param}.
The quantity $n_b=e_b-\Upsilon(e_b)$ is the first coordinate of $K_+$ at $e=e_b$, hence strictly increasing in $e_b$ and so in $b$, with limits $-1$ as $b\to0$ and $-\frac78$ as $b\to\frac12$.
Adding the endpoint $(-1,-1)$, which only $W$ attains by the proof of Theorem~\ref{thm:eta-nu}, identifies $K_+$ as the graph of $\Lambda$.
Uniqueness of the copulas attaining the points of $K_+$ and $K_-$ comes from Theorem~\ref{thm:eta-nu}: they are the transposes of the upper extremizers and the upper extremizers themselves, respectively.
\end{proof}

\Needspace{5\baselineskip}
\begin{remark}[Randomization and a by-product]
\label{rem:randomization}
\leavevmode
\begin{enumerate}[label=(\alph*),leftmargin=2em]
\item
For $w\in(0,\frac12)$, the identity $\nu(A_w)-\eta(A_w)=2(1-w)w^3$ still holds, yet the uniqueness statement in Theorem~\ref{thm:eta-nu} rules these copulas out as extremizers: for $\eta\in(-1,-\frac34)$, the boundary belongs to $B_b$ instead.
For this copula, the conditional law of $Z$ given a leading rank $X=x\in(1-b,1]$ consists of the two atoms in \eqref{eq:Ba-branches}.
Their branches are fixed by the equality set of Proposition~\ref{prop:dual-B}, their probabilities by the uniform marginals.

The first-order condition links the two branches: wherever the first potential is differentiable, a contact point satisfies $\partial_xd_\kappa(x,z)=\varphi'(x)$.
This equation is quadratic in $z$, so it has at most two solutions, and they sum to $2x/\kappa$, as in \eqref{eq:branch-sum}.
For $\kappa=\kappa_b>2$, the unique optimizer found above uses both of them on a set of positive measure.
The transpose $B_b^\top$, on the other hand, is an ordinary graph copula for the map $R_b$, so the extremal coupling is a graph over the second coordinate but not over the first.
\item
By \eqref{eq:reflection-transform} and \eqref{eq:eta-reflection},
\[
  \nu(C^\perp)-\eta(C^\perp)=\rho(C)-\eta(C).
\]
The map $C\mapsto C^\perp$ is an involution on $\CC$, so Corollary~\ref{cor:asymmetry} yields one more sharp bound, $|\rho(C)-\eta(C)|\le\frac{27}{128}$.
Only $A_{3/4}^\perp$, the shuffle of $M$ that is countermonotone on $[0,\frac34]$ and comonotone on $[\frac34,1]$, reaches the maximum $\rho-\eta=\frac{27}{128}$, and only its survival copula $(A_{3/4}^\top)^\perp$ reaches the minimum.
Shuffles of this kind showed in \cite[Ex.~4]{genest2003blest} that $\eta$ violates one of Scarsini's axioms.
\end{enumerate}
\end{remark}

\bibliographystyle{plainnat}
\bibliography{references}

\end{document}